\documentclass[12pt,a4paper]{amsart}
\usepackage{amsmath}
\usepackage{drpack}
\usepackage[english]{babel}
\usepackage{verbatim}
\usepackage[shortlabels]{enumitem}
\usepackage{color}
\usepackage{tikz-cd}

\newtheoremstyle{mio}%
	{}{} % spazio sopra e sotto%
	{\itshape}{} % corpo del testo, indentation
	{\bfseries}{.}{ } % titolo del teorema: tipo di testo, divisore, spaziatura
	{#1 #2\thmnote{~\mdseries(#3)}} % formattazione nota
\theoremstyle{mio}
\newtheorem{teor}{Theorem}[section]
\newtheorem{cor}[teor]{Corollary}
\newtheorem{prop}[teor]{Proposition}
\newtheorem{lemma}[teor]{Lemma}
\newtheorem{defin}[teor]{Definition}

\newtheoremstyle{definition2}%
	{}{} % spazio sopra e sotto%
	{}{} % corpo del testo, indentation
	{\bfseries}{.}{ } % titolo del teorema: tipo di testo, divisore, spaziatura
	{#1 #2\thmnote{\mdseries~ #3}} % formattazione nota
\theoremstyle{definition2}

\newtheorem{oss}[teor]{Remark}

\newcommand{\trasposta}{\intercal}
\newcommand{\Int}{\mathrm{Int}}
\newcommand{\caratt}{\mathfrak{J}}
\newcommand{\polya}{\mathrm{Po}}
\newcommand{\markov}{\mathcal{C}}

\title[Density of solutions and Markov chains]{Calculating the density of solutions of equations related to the P\'olya-Ostrowski group through Markov chains}
\author{Dario Spirito}
\email{spirito@mat.uniroma3.it}
\address{Dipartimento di Matematica e Fisica, Universit\`a degli Studi ``Roma Tre'', Roma, Italy}

\begin{document}
\begin{abstract}
Motivated by a problem in the theory of integer-valued polynomials, we investigate the natural density of the solutions of equations of the form $\theta_uu_q(n)+\theta_ww_q(n)+\theta_2\frac{n(n+1)}{2}+\theta_1n+\theta_0\equiv 0\bmod d$, where $d,q\geq 2$ are fixed integers, $\theta_u,\theta_w,\theta_2,\theta_1,\theta_0$ are parameters and $u_q$ and $w_q$ are functions related to the $q$-adic valuations of the numbers between 1 and $n$. We show that the number of solutions of this equation in $[0,N)$ satisfies a recurrence relation, with which we can associate to any pair $(d,q)$ a stochastic matrix and a Markov chain. Using this interpretation, we calculate the density for the case $\theta_u=\theta_2=0$ and for the case $\theta_u=1$, $\theta_w=\theta_2=\theta_1=0$ and either $d|q$ or $d$ and $q$ are coprime.
\end{abstract}

\maketitle

\section{Introduction}
Let $q\geq 2$ be a positive integer. Following \cite{IntD} and \cite{elliott-polya}, we define the following three functions on $\insN$:
\begin{equation*}
v_q(n):=\begin{cases}
0 & \text{if~}n=0\\
\max\{k\inN\mid q^k\text{~divides}~n\} & \text{if~}n>0;\\
%=q\text{-adic valuation of~}n 
\end{cases}\eqlnn
w_q(n):=\sum_{i=0}^nv_q(i)=\sum_{k\geq 1}\left\lfloor\frac{n}{q^k}\right\rfloor;\eqlnn
u_q(n):=\sum_{i=0}^nw_q(i).
\end{equation*}

The functions $w_q(n)$ and $u_q(n)$ arise naturally during the study of the \emph{P\'olya-Ostrowski group} $\polya(D)$ of a Dedekind domain $D$, a subgroup of the class group of $D$ closely related to the ring of integer-valued polynomials $\Int(D)$ over $D$. In particular, $\polya(D)$ is linked with the module structure of two sequences of $D$-modules, namely the sequence of \emph{characteristic ideals} $\caratt_n(D)$ of $D$ and the sequence of the subsets $\Int_n(D)$ of $\Int(D)$ formed by the polynomials $f\in\Int(D)$ of degree at most $n$. More precisely, $\caratt_n(D)$ and $\Int_n(D)$ are free if and only if, respectively, $w_q(n)\equiv 0\bmod d$ and $u_q(n)\equiv 0\bmod d$ for every $q$, where $d$ is the order of the ideal $\Pi_q$ (defined as the product of the maximal ideals $\mathfrak{m}$ of $D$ such that $|D/\mathfrak{m}|=q$) in the class group of $D$. See Section \ref{sect:Int} for a more detailed explanation.

This ideas led Elliott \cite{elliott-polya} to study the equations $w_q(n)\equiv x\bmod d$ and $u_q(n)\equiv x\bmod d$, where $d$ and $x$ are integers; in particular, he was interested in the \emph{density} of the set of solutions, where the \emph{(natural) density} of a subset $T\subseteq\insN$ is the limit
\begin{equation*}
\delta(T):=\lim_{N\to\infty}\frac{|\{n\in T\mid 0\leq n<N\}|}{N}
\end{equation*}
(provided that it exists). Through a mixture of special cases and experimental evidence, he conjectured \cite[Conjecture 2.2]{elliott-polya} that the density of the solutions always exists and is rational, and that if $x=0$ then it is at least $1/d$.

In this paper, we consider the more general equation
\begin{equation}\label{eq:totale}
\theta_uu_q(n)+\theta_ww_q(n)+\theta_2\frac{n(n+1)}{2}+\theta_1n+\theta_0\equiv 0\bmod d,
\end{equation}
where $\theta_u,\theta_w,\theta_2,\theta_1,\theta_0$ are integer coefficients; this form appears when trying to express the function $u_q(aq+\lambda)$ in terms of $u_q(a)$.

Our starting point is the possibility of expressing the number of solutions of \eqref{eq:totale} in $[0,qN)$ in terms of the number of solutions in $[0,N)$ of equations of the same form, but with different coefficients (Proposition \ref{prop:trasf-uq}). While we are not able to prove that the density $\delta(q,d;\theta_u,\theta_w,\theta_2,\theta_1,\theta_0)$ of the solutions of \eqref{eq:totale} exists for every choice of $q$, $d$ and the coefficients, we can use this recurrence relation to associate to the equation \eqref{eq:totale} (for any fixed $q$ and $d$) a stochastic matrix $P:=P(q,d)$ and a Markov chain, studying which we can calculate these densities in several cases.

More precisely, fix $q$ and $d$. Suppose that $\delta(q,d;\theta_u,\theta_w,\theta_2,\theta_1,\theta_0)$ exists for every choice of $\theta_u,\theta_w,\theta_2,\theta_1,\theta_0$. We prove that:
\begin{itemize}
\item (Theorem \ref{teor:rational}) every $\delta(q,d;\theta_u,\theta_w,\theta_2,\theta_1,\theta_0)$ is rational;
\item (Theorem \ref{teor:thetaw=0})
\begin{equation*}
\delta(q,d;0,\psi,0,\theta,x)=\begin{cases}
\inv{d}\gcd(\psi,\theta,d) & \text{if~}\gcd(\psi,\theta,d)|\gcd(x,d)\\
0 & \text{otherwise};
\end{cases}
\end{equation*}
\item (Theorem \ref{teor:d|q}) if $d|q$, then ($\varphi$ is the Euler function)
\begin{equation*}
\delta(q,d;1,0,0,0,x)=\inv{d^2}\sum_{f|\gcd(x,d)}f\cdot\varphi\left(\frac{d}{f}\right);
\end{equation*}
\item (Theorem \ref{teor:dqcoprimi}) if $d$ and $q$ are coprime and $\theta_u$ is coprime with $d$, then
\begin{equation*}
\delta(q,d;\theta_u,\theta_w,0,0,x)=\inv{d}.
\end{equation*}
\end{itemize}

Section \ref{sect:Int} translates the result obtained back to the setting of integer-valued polynomials.

\section{The general transformation}
Our first step is expressing $w_q(aq+\lambda)$ and $u_q(aq+\lambda)$ as a function of $a$ and $\lambda$.
\begin{lemma}\label{lemma:wqaq}
Let $q$ be a positive integer, and let $\lambda\in\{0,\ldots,q-1\}$. Then,
\begin{equation*}
w_q(aq+\lambda)=w_q(a)+a
\end{equation*}
for every $a\inN$.
\end{lemma}
\begin{proof}
Since $v_q(i)=0$ if $i$ is not a multiple of $q$, we have
\begin{equation*}
w_q(aq+\lambda)=\sum_{i=0}^{aq+\lambda}v_q(i)=\sum_{j=0}^av_q(jq).
\end{equation*}
Moreover, $v_q(jq)=1+v_q(j)$, and thus
\begin{equation*}
w_q(aq+\lambda)=\sum_{j=0}^a(1+v_q(j))=a+\sum_{j=0}^av_q(j)=a+w_q(a),
\end{equation*}
as claimed.
\end{proof}

\begin{lemma}\label{lemma:uqaq}
Let $q$ be a positive integer, and let $\lambda\in\{0,\ldots,q-1\}$. Then,
\begin{equation*}
\begin{split}
u_q(aq+\lambda) & =qu_q(a)+(\lambda+1-q)w_q(a)+\frac{q}{2}a^2+\left(\lambda+1-\frac{q}{2}\right)a=\\
& =qu_q(a)+	(\lambda+1-q)w_q(a)+q\frac{a(a+1)}{2}+(\lambda+1-q)a
\end{split}
\end{equation*}
for every $a\inN$.
\end{lemma}
\begin{proof}
We start by calculating $u_q(aq-1)$. We have
\begin{equation*}
u_q(aq-1)=\sum_{i=0}^{aq-1}w_q(i)=\sum_{b=0}^{a-1}\sum_{t=0}^{q-1}w_q(bq+t).
\end{equation*}
Since $w_q(bq+t)=w_q(bq)$ for every $t\in\{0,\ldots,q-1\}$, this implies that
\begin{equation*}
u_q(aq-1)=\sum_{b=0}^{a-1}qw_q(bq)=q\sum_{b=0}^{a-1}(w_q(b)+b)=qu_q(a-1)+q\frac{(a-1)a}{2}.
\end{equation*}
Hence, using again $w_q(bq+t)=w_q(bq)$, we have
\begin{equation*}
\begin{array}{rcl}
u_q(aq+\lambda) & = & u_q(aq-1)+(\lambda+1)w_q(aq)=\\[1em]
& = & qu_q(a-1)+q\frac{(a-1)a}{2}+(\lambda+1)(w_q(a)+a)=\\[1em]
& = & q(u_q(a)-w_q(a))+\frac{(a-1)aq}{2}+(\lambda+1)w_q(a)+(\lambda+1)a,
\end{array}
\end{equation*}
rearranging which we have our claim.
\end{proof}

The previous lemmas suggest to consider the more general equation
\begin{equation}\label{eq:tot-uq}
\theta_uu_q(n)+\theta_ww_q(n)+\theta_2\frac{n(n+1)}{2}+\theta_1n+\theta_0\equiv 0\bmod d,
\end{equation}
as $\theta_u,\theta_w,\theta_2,\theta_1,\theta_0$ vary in $\insZ$. Clearly, if $\theta'_u\equiv\theta_u\bmod d$, and analogously for $\theta'_w,\theta'_2,\theta'_1$ and $\theta'_0$, then $n$ is a solution of \eqref{eq:tot-uq} if and only if it is a solution of
\begin{equation*}
\theta'_uu_q(n)+\theta'_ww_q(n)+\theta'_2\frac{n(n+1)}{2}+\theta'_1n+\theta'_0\equiv 0\bmod d;
\end{equation*}
for this reason, we will sometimes consider Equation \eqref{eq:tot-uq} as having the coefficients $\theta_u,\theta_w,\theta_2,\theta_1,\theta_0$ in $\insZ/d\insZ$; this should not cause confusion.

Note that, if we were using $n^2$ instead of $n(n+1)/2$, we may need to consider also half-integer values of $\theta_2$ and $\theta_1$, and the situation may become troublesome when $d$ is even.

Let now $\mathbf{s}:=(\theta_u,\theta_w,\theta_2,\theta_1,\theta_0)\inZ^5$. For any $A\inN$, we denote by $\gamma(A,q,d;\theta_u,\theta_w,\theta_2,\theta_1,\theta_0)$, or by $\gamma(A,q,d;\mathbf{s})$, the number of natural numbers $n<A$ that are solution of \eqref{eq:tot-uq}.

\begin{prop}\label{prop:trasf-uq}
Let $q,d\geq 2$ be positive integers. For every $A\inN$, and every $\mathbf{s}\inZ^5$, we have
\begin{equation}\label{eq:gamma-sumlambda}
\gamma(qA,q,d;\mathbf{s})=\sum_{\lambda=0}^{q-1}\gamma(A,q,d;\mathbf{s}M_\lambda)
\end{equation}
where
\begin{equation*}
M_\lambda:=\begin{pmatrix}
q & \lambda-q+1 & q & \lambda-q+1 & 0\\
0 & 1 & 0 & 1 & 0\\
0 & 0 & q^2 & \lambda q-\frac{q(q-1)}{2} & \frac{\lambda(\lambda+1)}{2}\\
0 & 0 & 0 & q & \lambda\\
0 & 0 & 0 & 0 & 1
\end{pmatrix}
\end{equation*}
for every $\lambda$.
\end{prop}

\begin{proof}
Let $\mathbf{s}:=(\theta_u,\theta_w,\theta_2,\theta_1,\theta_0)$. For any $\lambda\in\{0,\ldots,q-1\}$, let $\gamma_{(\lambda)}(A,q,d;\mathbf{s})$ be the number of solutions to Equation \eqref{eq:tot-uq} that are smaller than $A$ and that are congruent to $\lambda$ modulo $q$.

Each row of $M_\lambda$ is the expansion of $u_q(n),w_q(n),\frac{n(n+1)}{2},n,1$ in terms of $u_q(a),w_q(a),\frac{a(a+1)}{2},a,1$ when $n=aq+\lambda$: indeed, for $u_q$ and $w_q$ it follows from Lemmas \ref{lemma:wqaq} and \ref{lemma:uqaq}, and it is obvious for $n$ and 1. Moreover,
\begin{equation*}
\frac{(aq+\lambda)(aq+\lambda+1)}{2}=\frac{q^2}{2}a^2+\frac{q^2}{2}a-\frac{q^2}{2}a+\left(q\lambda+\frac{q}{2}\right)a+\frac{\lambda(\lambda+1)}{2}
\end{equation*}
which gives the third row of $M_\lambda$ after rearrangement.

Now $aq+\lambda<Aq$ if and only if $a<A$; this means exactly that
\begin{equation*}
\gamma_{(\lambda)}(qA,q,d;\mathbf{s})=\gamma(A,q,d;\mathbf{s}M_\lambda).
\end{equation*}
Summing over $\lambda$ we get the claim.
\end{proof}

As remarked before the statement of the proposition, we can consider $\mathbf{s}$ to be an element of $(\insZ/d\insZ)^5$ instead of $\insZ^5$; in particular, we can define $\gamma(A,q,d;\mathbf{s})$ even with $\mathbf{s}\in(\insZ/d\insZ)^5$, and Proposition \ref{prop:trasf-uq} carries over without problems, with the only difference that each $M_\lambda$ must be considered as a matrix over $\insZ/d\insZ$.

This convention is useful because it makes the space of possible $\mathbf{s}$ finite, and in particular it allows to rearrange Equation \eqref{eq:gamma-sumlambda} in matrix form. Indeed, for every $\mathbf{s},\mathbf{t}\in(\insZ/d\insZ)^5$, let $\mu(\mathbf{s},\mathbf{t})$ be the number of $\lambda\in\{0,\ldots,q-1\}$ such that $\mathbf{t}=\mathbf{s}M_\lambda$. Then, for every $\mathbf{s}\in(\insZ/d\insZ)^5$ we have the finite sum
\begin{equation*}
\gamma(qA,q,d;\mathbf{s})=\sum_{\mathbf{t}\in(\insZ/d\insZ)^5}\mu(\mathbf{s},\mathbf{t})\gamma(A,q,d;\mathbf{t})
\end{equation*}
Let now 
\begin{equation*}
\widetilde{\gamma}(A,q,d;\mathbf{s}):=\frac{\gamma(A,q,d;\mathbf{s})}{A},
\end{equation*}
and let $\boldsymbol{\widetilde{\gamma}}(A,q,d)$ be the column vector composed of the $\widetilde{\gamma}(A,q,d;\mathbf{s})$, as $\mathbf{s}$ ranges in $(\insZ/d\insZ)^5$. Then, the previous equality can be written as
\begin{equation}\label{eq:widetilde-P}
\boldsymbol{\widetilde{\gamma}}(qA,q,d)=P(q,d)\boldsymbol{\widetilde{\gamma}}(A,q,d)
\end{equation}
where $P(q,d):=(\mu(\mathbf{s},\mathbf{t})/q)_{\mathbf{s},\mathbf{t}}$ is a (rational) matrix of order $d^5$. This matrix is a \emph{stochastic matrix}, i.e., each entry is nonnegative and the sum of each of its row is 1: indeed, the sum of $\mu(\mathbf{s},\mathbf{t})$, as $\mathbf{s}$ is fixed and $\mathbf{t}$ varies, must be $q$, since for each $\lambda$ there is a $\mathbf{t}$ such that $\mathbf{t}=\mathbf{s}M_\lambda$.

We introduce the following definition.
\begin{defin}\label{defin:density}
Let $q,d\geq 2$ be positive integers. The \emph{density of solutions} for $\mathbf{s}\inZ^5$ (or $\mathbf{s}\in(\insZ/d\insZ)^5$) with respect to $q$ and $d$ is
\begin{equation*}
\delta(q,d;\mathbf{s}):=\lim_{N\to\infty}\widetilde{\gamma}(N,q,d;\mathbf{s})= \lim_{N\to\infty}\frac{\gamma(N,q,d;\mathbf{s})}{N},
\end{equation*}
provided that the limit exists; if $q$ and $d$ are clear from the context, we also write $\delta(\mathbf{s})$ for $\delta(q,d;\mathbf{s})$. The column vector $(\delta(\mathbf{s}))_{\mathbf{s}}$ is called the \emph{vector of densities} of the solutions of Equation \eqref{eq:tot-uq} and is denoted by $\boldsymbol{\delta}(q,d)$ (or simply $\boldsymbol{\delta}$).
\end{defin}

Fix now the first four coefficients $\theta_u,\theta_w,\theta_2,\theta_1$. If the density $\delta(q,d;\theta_u,\theta_w,\theta_2,\theta_1,x)$ exists for every $x\inZ$, we say that the function
\begin{equation*}
f:n\mapsto \theta_uu_q(n)+\theta_ww_q(n)+\theta_2\frac{n(n+1)}{2}+\theta_1n
\end{equation*}
\emph{has a limit distribution modulo $d$} (and we call the assignment $x\mapsto\delta(q,d;\theta_u,\theta_w,\theta_2,\theta_1,x)$ the \emph{limit distribution}). If the densities are all equal (and so are equal to $1/d$), we say that $f$ is \emph{uniformly distributed modulo $d$} (see e.g. \cite{unifdistrib}).

\begin{lemma}\label{lemma:autoval}
Let $M$ be a square matrix of order $n$ over $\insC$; suppose that each eigenvalue $\lambda$ of $M$ satisfies $|\lambda|=1$, $\lambda\neq 1$. If $\mathbf{v}$ is a vector such that $M^k\mathbf{v}$ has a limit when $k\to\infty$, then $\mathbf{v}=0$.
\end{lemma}
\begin{proof}
By conjugation, we can suppose that $M=(m_{ij})_{i,j}$ is an upper triangular matrix.

Suppose $\mathbf{v}=(v_1,\ldots,v_n)\neq 0$, and let $t$ be the biggest $i$ such that $v_i\neq 0$. Then, the $t$-th component of $M^k\mathbf{v}$ is equal to $m_{tt}^kv_t$; in particular, $m_{tt}^kv_t$ has a limit as $k\to\infty$. However, $m_{tt}$ is an eigenvalue of $M$, and thus it is a complex number of norm 1 different from 1; hence, $m_{tt}^k$ does not have a limit. This would imply $v_t=0$, against our choice of $t$. The claim is proved.
\end{proof}

\begin{teor}\label{teor:rational}
Let $q,d\geq 2$ be positive integers. If the vector of densities $\boldsymbol{\delta}(q,d)$ exists, then it is a right eigenvector of $P(q,d)$ with eigenvalue $1$, and all its entries are rational.
\end{teor}
\begin{proof}
For every $k\inN$, let $\boldsymbol{\widetilde{\gamma}}_k$ be the column vector whose entries are $\widetilde{\gamma}(q^k,q,d;\mathbf{s})$, as $\mathbf{s}$ ranges in $(\insZ/d\insZ)^5$. Then, \eqref{eq:widetilde-P} becomes
\begin{equation*}
\boldsymbol{\widetilde{\gamma}}_k=P\boldsymbol{\widetilde{\gamma}}_{k-1}=P^k\boldsymbol{\widetilde{\gamma}}_0.
\end{equation*}
Clearly, if $\boldsymbol{\widetilde{\gamma}}_k$ has a limit for $k\to\infty$ then it must be $\boldsymbol{\delta}:=\boldsymbol{\delta}(q,d)$; in particular, the first equality of the previous equation becomes
\begin{equation*}
\boldsymbol{\delta}=P\boldsymbol{\delta},
\end{equation*}
and thus $\boldsymbol{\delta}$ is an eigenvector of $P$ with eigenvalue 1.

Moreover, the existence of $\boldsymbol{\delta}$ implies that $P^k\boldsymbol{\widetilde{\gamma}}_0$ has a limit as $k\to\infty$. Being $P:=P(q,d)$ a stochastic matrix with rational entries, the algebraic and geometric multiplicities of its eigenvalue 1 coincide (see e.g. \cite[Section 9.4, Fact 1(b)]{handbook-LinearAlgebra} or \cite[p.696]{MatrixAnalysis}), and we can find a rational matrix $A$ such that $A^{-1}PA$ is a block matrix
\begin{equation*}
N:=A^{-1}PA=\begin{pmatrix}I & 0 & 0\\
0 & R & 0\\
0 & 0 & Q\end{pmatrix},
\end{equation*}
where $I$ is the identity matrix, the eigenvalues of $R$ have norm 1 but are different from 1 and the norm of each eigenvalue of $Q$ is smaller than 1. The limit $P^k\boldsymbol{\widetilde{\gamma}}_0\longrightarrow\boldsymbol{\delta}$ can we rewritten as $N^k(A^{-1}\boldsymbol{\widetilde{\gamma}}_0)\longrightarrow A^{-1}\boldsymbol{\delta}$. Let $\mathbf{v}:=A^{-1}\boldsymbol{\widetilde{\gamma}}_0$; then,
\begin{equation*}
N^k\mathbf{v}=\begin{pmatrix}I & 0 & 0\\
0 & R^k & 0\\
0 & 0 & Q^k\end{pmatrix}\begin{pmatrix}\mathbf{v}_1\\ \mathbf{v_2}\\ \mathbf{v_3}\end{pmatrix}=\begin{pmatrix}\mathbf{v}_1\\ R^k\mathbf{v_2}\\ Q^k\mathbf{v_3}\end{pmatrix},
\end{equation*}
where $\mathbf{v}_1$, $\mathbf{v}_2$ and $\mathbf{v}_3$ are subvectors of $\mathbf{v}$ of appropriate length. The existence of the limit implies that both $R^k\mathbf{v_2}$ and $Q^k\mathbf{v_3}$ have limit as $k\to\infty$. Since $R$ satisfies the hypothesis of Lemma \ref{lemma:autoval}, we have $\mathbf{v}_2=0$; on the other hand, by construction, $Q^k\longrightarrow 0$ \cite[p.617]{MatrixAnalysis}, and thus $Q^k\mathbf{v}_3\longrightarrow 0$.

Both $A$ and $\boldsymbol{\widetilde{\gamma}}_0$ have rational entries (since $\gamma(1,q,d;\mathbf{s})$ is either 1 or 0). Hence, $\mathbf{v}$ has rational entries, and by the previous reasoning so does the limit of $N^k\mathbf{v}$, i.e., $A^{-1}\boldsymbol{\delta}$; therefore, $\boldsymbol{\delta}$ has rational entries, as claimed.
\end{proof}

The fact that the entries are rational supports part of \cite[Conjecture 2.2(1)]{elliott-polya}.

As observed before Definition \ref{defin:density}, $P:=P(q,d)$ is a stochastic matrix: hence, we can interpret it in a probabilistic way. A (discrete) \emph{Markov chain} $\mathcal{M}$ is a family of random variables $\{X_n\}_{n\inN}$, whose range is a finite set $S:=\{s_1,\ldots,s_k\}$ (called the \emph{state space} of $\mathcal{M}$), such that 
\begin{equation*}
P(X_{n+1}=t_{n+1}|X_n=t_n,X_{n-1}=t_{n-1},\ldots,X_0=t_0)=P(X_{n+1}=t_{n+1}|X_n=t_n)
\end{equation*}
for all $n$ and all $t_1,\ldots,t_{n+1}\in S$. If, furthermore, the probability $P(X_{n+1}=t_{n+1}|X_n=t_n)$ of going from $t_n$ to $t_{n+1}$ does not depend on $n$, then $\mathcal{M}$ is said to be \emph{time-homogeneous}, and the matrix $M:=(m_{ij})_{i,j}$, where $m_{ij}:=P(X_{n+1}=s_j|X_n=s_i)$, is a stochastic matrix, called the \emph{transition matrix} of $\mathcal{M}$.

Conversely, to every stochastic matrix $M=(m_{ij})_{i,j}$ of order $k$ is associated a discrete, time-homogeneous Markov chain $\mathcal{M}$ on a state space $S=\{s_1,\ldots,s_k\}$ of cardinality $k$, such that $P(X_{n+1}=s_j|X_n=s_i)=m_{ij}$ for all $n$, that is, such that $M$ is the transition matrix of $\mathcal{M}$. When $M$ and $\mathcal{M}$ are linked in this way, we call $\mathcal{M}$ the Markov chain \emph{represented} by $M$. See e.g. \cite[Section 8.4]{MatrixAnalysis} for further details.

When $M=P(q,d)$, we denote the Markov chain arising in this way by $\markov(q,d)$; more explicitly, $\markov(q,d)$ is the Markov chain such that the probability of going from state $\mathbf{s}$ to state $\mathbf{t}$ is $\mu(\mathbf{s},\mathbf{t})/q$.

%Note that $\boldsymbol{\delta}$ is \emph{not} a stationary distribution of this chain, because such a distribution $\boldsymbol{\pi}$ is a \emph{left} eigenvector of $P$, i.e., $\boldsymbol{\pi}=\boldsymbol{\pi}P$.

Let $\mathcal{M}$ be a Markov chain with state space $X$, and let $i,j\in X$. We say that $j$ is \emph{reachable} from $i$ (or that $i$ \emph{leads to} $j$) if there is a $k\inN$ such that $P(X_{n+k}=j|X_n=i)>0$, that is, if the probability of going from $i$ to $j$ in $k$ steps is positive. We say that $i$ and $j$ are \emph{communicating} (and we write $i\leftrightarrow j$) if $i$ is reachable from $j$ and $j$ is reachable from $i$. The relation ``$\leftrightarrow$'' is an equivalence relation; if $C$ is an equivalence class, we say that $C$ is \emph{ergodic} if, when $i\in C$ and $j$ is reachable from $i$, then also $j\in C$, that is, if one the chain arrives in $C$ then it cannot leave $C$; equivalently, if $P(X_{n+k}=l|X_n=i)=0$ for all $l\notin C$ and all $k\inN$. A state is \emph{ergodic} if the equivalence class it belongs to is an ergodic class.
\begin{prop}\label{prop:markov-autovett}
Let $P$ be a stochastic matrix, and let $\mathbf{v}:=(v_1,\ldots,v_n)^\trasposta$ be such that $\mathbf{v}=P\mathbf{v}$. If $i$ and $j$ belong to the same ergodic class, then $v_i=v_j$.
\end{prop}
\begin{proof}
This is essentially a consequence of the Perron-Frobenius theorem (see e.g. \cite[Section 9.2, Fact 5]{handbook-LinearAlgebra} or \cite[Chapter 8]{MatrixAnalysis}). Let $C_1,\ldots,C_t$ be the ergodic classes of the Markov chain associated to $P$. By \cite[Section 9.4, Fact 1(g)]{handbook-LinearAlgebra}, the space $V$ of right eigenvectors of $P$ with eigenvalue 1 has dimension $t$, and there is a basis $\underline{\mathbf{u}}:=\{\mathbf{u}^{C_1},\ldots,\mathbf{u}^{C_t}\}$ of $V$ such that, if $a$ is an ergodic state and $C_k$ is its equivalence class, then the $a$-th component of $\mathbf{u}^{C_l}$ is 1 if $l=k$ and 0 otherwise. In particular, if both $i,j\in C_k$ and $\mathbf{v}\in V$, then $v_i$ and $v_j$ are equal to the coefficients of $\mathbf{u}^{C_k}$ along the basis $\underline{\mathbf{u}}$ of $V$. In particular, $v_i=v_j$.
\end{proof}

\section{The case $\theta_u=\theta_2=0$}\label{sect:wq}
A function $f:\insN\longrightarrow\insR$ is said to be \emph{$q$-additive} if $f(0)=0$ and
\begin{equation*}
f(n)=\sum_{j\geq 0}f(a_{q,j}(n)q^j)\quad\text{for}\quad n=\sum_{j\geq 0}a_{q,j}(n)q^j,
\end{equation*}
where $a_{q,j}(n)\in\{0,\ldots,q-1\}$ are the digits of $n$ in base $q$. The prototype of $q$-additive functions is the sum of digits of $n$ in base $q$, which we denote by $s_q(n)$.

By \cite[Exercise II.8 and Lemma II.4]{IntD}, we can write
\begin{equation*}
w_q(n):=\frac{n-s_q(n)}{q-1};
\end{equation*}
thus, for every $\psi,\theta\inZ$, the function
\begin{equation*}
\psi w_q(n)+\theta n=\frac{\theta(q-1)+\psi}{q-1}\cdot n-\frac{\psi}{q-1}\cdot s_q(n);
\end{equation*}
is $q$-additive. On the other hand, note that $u_q(n)$ is not $q$-additive, since (for example) $u_q(q+1)=2\neq 1=u_q(q)+u_q(1)$. This point of view yields the following, the second part of which is another partial answer to \cite[Corollary 2.2(1)]{elliott-polya}.
\begin{teor}\label{teor:qadd}
Let $d,q\geq 2$ be positive integers. Then:
\begin{enumerate}[(a)]
\item\label{teor:qadd:wq} for every $\theta_w,\theta_1,\theta_0\inZ$, the density $\delta(q,d;0,\theta_w,0,\theta_1,\theta_0)$ exists;
\item\label{teor:qadd:dqn} if $d|q^n$ for some $n$, then the vector of densities $\boldsymbol{\delta}(q,d)$ exists.
\end{enumerate}
\end{teor}
\begin{proof}
\ref{teor:qadd:wq} By \cite[Theorem 1.1]{limitdistrib-qadd}, every $q$-additive function with integer values has a limit distribution modulo $d$, for every $d$; in particular, so does $\theta_w w_q(n)+\theta_1 n$, and the density $\delta(q,d;0,\theta_w,0,\theta_1,\theta_0)$ exists.

\ref{teor:qadd:dqn} Let $\lambda_1,\ldots,\lambda_n\in\{0,\ldots,q-1\}$, and let $\mathbf{s}\in(\insZ/d\insZ)^5$. Applying repeatedly Proposition \ref{prop:trasf-uq}, we have
\begin{equation*}
\frac{\gamma(q^nN,q,d;\mathbf{s})}{q^nN}=\sum_{\lambda_1,\ldots,\lambda_n=0}^{q-1} \inv{q^n}\frac{\gamma(N,q,d;\mathbf{s}M_{\lambda_1}\cdots M_{\lambda_n})}{N}.
\end{equation*}
The first and the third column of $M_{\lambda_1}\cdots M_{\lambda_n}$ (as integer matrices) are divisible by $q^n$; hence, they are 0 when reduced modulo $d$. It follows that the first and the third component of $\mathbf{s}M_{\lambda_1}\cdots M_{\lambda_n}$ are always 0.

Therefore, by the previous point, for each summand of the right hand side the limit (as $N\to\infty$) exists; it follows that so does the limit of the left hand side, i.e., the density $\delta(q,d;\mathbf{s})$ exists.
\end{proof}

The first part of the previous theorem is especially useful since, if we are interested in the distribution of the function $w_q$, we don't need to consider $\theta_u$ or $\theta_2$; that is, we can study the Markov chain limited to the subset of $(\insZ/d\insZ)^5$ where $\theta_u=\theta_2=0$. The next result calculates these densities.
\begin{teor}\label{teor:thetaw=0}
Let $q,d\geq 2$ be positive integers, and let $\psi\inZ$, $\psi\neq 0$. Then,
\begin{equation*}
\delta(q,d;0,\psi,0,\theta,x)=\begin{cases}
\inv{d}\gcd(\psi,\theta,d) & \text{if~}\gcd(\psi,\theta,d)|\gcd(x,d)\\
0 & \text{otherwise}
\end{cases}
\end{equation*}
for every $x\inZ$.
\end{teor}
\begin{proof}
By Theorem \ref{teor:qadd}\ref{teor:qadd:wq} the density exists.

We first note that, if $\gcd(\psi,\theta,d)$ does not divide $\gcd(x,d)$, then the equation $\psi w_q(n)+\theta n+x\equiv 0\bmod d$ cannot have solutions, so the density is 0.

Suppose that $g:=\gcd(\psi,\theta,d)$ divides $\gcd(x,d)$. In this case, $\psi w_q(n)+\theta n+x\equiv 0\bmod d$ is equivalent to $\frac{\psi}{g} w_q(n)+\frac{\theta}{g} n+\frac{x}{g}\equiv 0\bmod \frac{d}{g}$; therefore,
\begin{equation*}
\gamma(A,q,d;0,\psi,0,\theta,x)=\gamma\left(A,q,\frac{d}{g};0,	\frac{\psi}{g},0,\frac{\theta}{g},\frac{x}{g}\right).
\end{equation*}
Moreover, $\gcd(\psi/g,\theta/g,d/g)=1$; hence, it is enough to prove the claim in the case $\gcd(\psi,\theta,d)=1$.

The set $X:=\{(0,\psi,0,\theta,x):\psi,\theta,x\in\insZ/d\insZ\}$ is invariant by multiplication on the right by the $M_\lambda$; hence, the Markov chain $\markov(q,d)$ restricts to a chain $\markov'(q,d)$ on $X$, which can also be defined as the Markov chain having the transition matrix $Q:=(\mu'(\mathbf{s},\mathbf{t})/q)$, where $\mu'(\mathbf{s},\mathbf{t})$ is the number of $\lambda$ such that $\mathbf{s}N_\lambda=\mathbf{t}$ and
\begin{equation*}
N_\lambda:=\begin{pmatrix}
1 & 1 & 0\\
0 & q & \lambda\\
0 & 0 & 1
\end{pmatrix}
\end{equation*}
is just the submatrix of $M_\lambda$ relative to $\theta_w$, $\theta_1$ and $\theta_0$.

Let $k\geq 1$. Consider the stochastic matrix $Q^k$; then, the entry relative to $\mathbf{s},\mathbf{t}$ is $\mu_k(\mathbf{s},\mathbf{t})/q^k$, where $\mu_k(\mathbf{s},\mathbf{t})$ is the number of $k$-tuples $(\lambda_1,\ldots,\lambda_k)$ such that $\mathbf{s}M_{\lambda_1}\cdots M_{\lambda_k}=\mathbf{t}$. Consider the Markov chain $\markov'_k(q,d)$ represented by $Q^k$: then, the probability of going from $\mathbf{s}$ to $\mathbf{t}$ is equal to $\mu_k(\mathbf{s},\mathbf{t})/q^k$, which is equal to the probability of going from $\mathbf{s}$ to $\mathbf{t}$ in $k$ steps in $\markov'(q,d)$. 

We claim that, for $k\geq 1$,
\begin{equation*}
\begin{pmatrix}\psi & \theta & x\end{pmatrix}N_{\lambda_1}\cdots N_{\lambda_k}= \begin{pmatrix}\psi\\ q^k\theta+q^{k-1}\psi+q^{k-2}\psi+\cdots+q\psi+\psi\\ x+P_k(\lambda_1,\ldots,\lambda_k,\psi,\theta,q)\end{pmatrix}^\trasposta,
\end{equation*}
where 
\begin{equation*}
P_k(\lambda_1,\ldots,\lambda_k,\psi,\theta,q):=\lambda_1\theta+\lambda_2(q\theta+\psi)+\cdots+ \lambda_k(q^{k-1}\theta+q^{k-2}\psi+\cdots+\psi).
\end{equation*}
Indeed, this is clear for $k=1$ and follows easily by induction for arbitrary $k$.

Consider the succession $a_k(\theta):=q^k\theta+q^{k-1}\psi+q^{k-2}\psi+\cdots+q\psi+\psi$, and suppose there is a $k>1$ such that $a_k(\theta)\equiv\theta\bmod d$. Then, for any $\psi$ and $\theta$, the subset $X(\psi,\theta):=\{(\psi,\theta,x):x\in\insZ/d\insZ\}$ of $X$ is invariant by $\markov'_k(q,d)$, and thus we can restrict the Markov chain to $X(\psi,\theta)$; we claim that this chain is irreducible, i.e., that any state can be reached from any state.

The probability of going from $(\psi,\theta,x)$ to $(\psi,\theta,y)$ is nonzero if and only if there are $\lambda_1,\ldots,\lambda_k$ such that $y-x=P_k(\lambda_1,\ldots,\lambda_k,q,\psi,\theta,q)$; hence, $x$ and $y$ are communicating if and only if they belong to the same coset of the subgroup $G$ of $\insZ/d\insZ$ generated by $P_k(\lambda_1,\ldots,\lambda_k,\psi,\theta,q)$, as $\lambda_1,\ldots,\lambda_k$ vary in $0,\ldots,q-1$.

The group $G$ contains both $\theta$ and $\psi$, since $P_k(1,0,\ldots,0)=\theta$ and $P_k(0,1,0,\ldots,0)=q\theta+\psi$. Since $\gcd(\psi,\theta,d)=1$, this implies that $G=\insZ/d\insZ$, and thus any state can be reached from any $x$. By Proposition \ref{prop:markov-autovett}, $\delta(0,\psi,0,\theta,x)=\delta(0,\psi,0,\theta,y)$ for all $x,y$. However,
\begin{equation*}
\sum_{x=0}^{d-1}\delta(0,\psi,0,\theta,x)=1;
\end{equation*}
hence, $\delta(0,\psi,0,\theta,x)=1/d$ for all $x$, as claimed.

Suppose now that $a_k(\theta)\not\equiv\theta\bmod d$ for every $k$. Even in this case, there must be $m<m'$ such that $a_m(\theta)\equiv a_{m'}(\theta)\bmod d$; since $a_{i+j}(\theta)=a_i(a_j(\theta))$ for every $i,j,\theta$, this means that $a_k(a_m(\theta))\equiv a_m(\theta)\bmod d$, where $k:=m'-m$. By the previous part of the proof, $\delta(0,\psi,0,a_m(\theta),x)=1/d$ for every $x\inZ$. Since the set of the densities $\delta(0,\psi,0,\theta,x)$ is an eigenvector of $Q^m$, we have
\begin{equation*}
\delta(\mathbf{s})=\sum_{\mathbf{t}\in X}\frac{\mu_m(\mathbf{s},\mathbf{t})}{q^m}\delta(\mathbf{t}).
\end{equation*}
Now $\mu_m((0,\psi,0,\theta,x),(0,\psi',0,\theta',x))=0$ unless $\psi'=\psi$ and $\theta'=a_m(\theta)$; hence, if $\mathbf{s}:=(0,\psi,0,\theta,x)$ we have
\begin{equation*}
\delta(\mathbf{s})=\sum_{y=0}^{d-1}\frac{\mu_m(\mathbf{s},(0,\psi,0,a_m(\theta),y))}{q^m}\delta(0,\psi,0,a_m(\theta),y).
\end{equation*}
By the previous part of the proof, each $\delta(0,\psi,0,a_m(\theta),y)$ is equal to $1/d$; since the sum of all $\mu_m(\mathbf{s},\mathbf{t})$ is $q^m$, this means that $\delta(0,\psi,0,\theta,x)=1/d$, as claimed.
\end{proof}

\begin{cor}\label{cor:wqthetan}
Let $q,d\geq 2$ be positive integers. For every $\theta$, the map $n\mapsto w_q(n)+\theta n$ is uniformly distributed.
\end{cor}
\begin{proof}
It is enough to consider the case $\psi=1$ in the previous theorem.
\end{proof}

We remark that in some cases we can also prove the existence of the limit without using the theory of $q$-additive functions.
\begin{prop}\label{prop:wqmodq}
Let $q\geq 2$ be a positive integer and $x,\theta\inZ$. If $\theta$ is coprime with $q$, or if $q|\theta$, then the map $n\mapsto w_q(n)+\theta n$ is uniformly distributed modulo $q$.
\end{prop}
\begin{proof}
Suppose first that $\theta$ is coprime with $q$. In each block $\{qa,\ldots,qa+q-1\}$, $w_q(n)$ is constant; hence, the equation $w_q(n)+\theta n\equiv x\bmod q$ has a unique solution, namely $n=qa+r$ with $r\equiv\theta^{-1}(x-w_q(qa))\bmod q$. Hence, the number of solutions of $w_q(n)+\theta n\equiv x\bmod q$ in $[0,\ldots,N)$ is $N/d+O(1)$; i.e., $\gamma(N,q,q;0,1,0,\theta,x)=N/d+O(1)$. Dividing by $N$, the limit on the right hand side exists and is equal to $1/d$; hence, the same applies to the left hand side.

For $q|\theta$ (and it is enough to consider $\theta=0$), we note that
\begin{equation*}
\frac{\gamma(N,q,q;0,1,0,0,x)}{N}=q\frac{\gamma(N/q,q,q;0,1,0,1,x)}{N}+O(1)
\end{equation*}
and the right hand side goes to $1/d$ by the previous reasoning. The claim is proved. 
\end{proof}

\section{The case $\mathbf{s}=(1,0,0,0,x)$}\label{sect:uq}
The methods used in the proof of Theorem \ref{teor:thetaw=0} can also be applied to study the full equation \eqref{eq:tot-uq}; in particular, we shall be interested in the equation with $(\theta_u,\theta_w,\theta_2,\theta_1,\theta_0)=(1,0,0,0,x)=:\mathbf{s}_x$. We are not able to obtain a full picture of the situation, so we will concentrate on two special cases. Before analyzing them, we note that we can obtain a lower limit for the density of the solutions.
\begin{prop}\label{prop:gamma-phi}
Let $d,q$ be positive integers and let $x\inZ$. Then,
\begin{equation*}
\liminf_{N\to\infty}\frac{\gamma(N,q,d;\mathbf{s}_x)}{N}\geq\frac{\varphi(d)}{dq}\left\lfloor\frac{q}{d}\right\rfloor,
\end{equation*}
where $\varphi$ is the Euler function.
\end{prop}
\begin{proof}
Consider the blocks $\{aq,aq+1,\ldots,aq+q-1\}$ of $q$ consecutive natural numbers, starting from a multiple of $q$. In any block, the map $n\mapsto w_q(n)$ is constant, and thus the map $n\mapsto u_q(n)$ is of constant difference; in particular,
\begin{equation*}
u_q(aq+r)=u_q(aq)+rw_q(aq).
\end{equation*}
If $w_q(aq)$ is coprime with $d$, then $u_q(aq)+rw_q(aq)$ passes through every residue class modulo $d$, as $d$ goes from 0 to $d-1$; hence, the equation $u_q(n)\equiv x\bmod d$ has at least $\left\lfloor q/d\right\rfloor$ solutions into each block $\{aq,aq+1,\ldots,aq+q-1\}$.

Let now $N$ be an integer, and divide $\{1,\ldots,N\}$ into blocks of length $q$. By Corollary \ref{cor:wqthetan}, in approximately $\varphi(d)/d$ of these blocks $w_q(n)$ is coprime with $d$; hence,
\begin{equation*}
\gamma(N,q,d;1,0,0,0,x)\geq\frac{N}{q}\frac{\varphi(d)}{d}\left\lfloor\frac{q}{d}\right\rfloor+O(1).
\end{equation*}
Dividing by $N$ and taking the limit inferior we get our claim.
\end{proof}

This result is far from being completely satisfying (for example, it says nothing when $q<d$). However, it sometimes hits what is actually the real density, as the following theorem shows.

\begin{teor}\label{teor:d|q}
Let $d,q\geq 2$ be positive integers, and suppose that $d|q$. Then, for every $x\in\insZ$,
\begin{equation*}
\delta(q,d;1,0,0,0,x)=\inv{d^2}\sum_{f|\gcd(x,d)}f\cdot\varphi\left(\frac{d}{f}\right).
\end{equation*}
\end{teor}
This result can be seen as a generalization of \cite[Proposition 3.4]{elliott-polya}.
\begin{proof}
Let $\mathbf{s}_x:=(1,0,0,0,x)$, with $x\in\insZ/d\insZ$; by Theorem \ref{teor:qadd}\ref{teor:qadd:dqn}, the density $\delta(q,d;\mathbf{s}_x)$ exists. We have $\mathbf{s}_xM_\lambda=(0,\lambda+1,0,\lambda+1,x)$; hence, applying Theorem \ref{teor:thetaw=0}, we have
\begin{equation*}
\delta(q,d;\mathbf{s}_x)=\inv{q}\sum_{\substack{\lambda=1,\ldots,q\\\lambda|\gcd(d,x)}}\frac{\gcd(\lambda,d)}{d}.
\end{equation*}
However, since $d|q$, the summand for $\lambda$ is equal to the summand for $\lambda+d$. Hence, the previous formula is equal to
\begin{equation*}
\inv{q}\frac{q}{d}\sum_{f|\gcd(x,d)}\frac{f}{d}|\{t\in\{1,\ldots,d\}:\gcd(t,d)=f\}|= \inv{d^2}\sum_{f|\gcd(x,d)}f\cdot\varphi\left(\frac{d}{f}\right),
\end{equation*}
as claimed.
\end{proof}

\begin{oss}
If $\gcd(d,x)=1$, then the theorem gives $\displaystyle{\delta(q,d;\mathbf{s})=\frac{\varphi(d)}{d^2}}$, exactly the limit inferior obtained in Proposition \ref{prop:gamma-phi}.
\end{oss}

The second case we consider is when $d$ and $q$ are coprime; it can be seen as a generalization of \cite[Proposition 3.6]{elliott-polya}, albeit with a stronger hypothesis (since we need that all the densities exist). We denote by $\mathcal{U}(\insZ/d\insZ)$ the set of units of $\insZ/d\insZ$.
\begin{teor}\label{teor:dqcoprimi}
Let $q,d\geq 2$ be coprime positive integers. Suppose that, for every $\mathbf{s}\inZ^5$, the density $\delta(q,d;\mathbf{s})=\delta(\mathbf{s})$ exists. Then, for every $\theta_w,x\in\insZ/d\insZ$ and each $\theta_u\in\mathcal{U}(\insZ/d\insZ)$, we have
\begin{equation*}
\delta(q,d;\theta_u,\theta_w,0,0,x)=\inv{d}.
\end{equation*}
In particular, the map $n\mapsto u_q(n)$ is uniformly distributed modulo $d$.
\end{teor}
\begin{proof}
The proof is similar to the proof of Theorem \ref{teor:thetaw=0}. Note first that we can suppose $\theta_u=1$, since $\gamma(N,q,d;\mathbf{s})=\gamma(N,q,d;u\mathbf{s})$ for every $u\in\mathcal{U}(\insZ/d\insZ)$. Since all the limits exist, by Theorem \ref{teor:rational} the density vector $\boldsymbol{\delta}$ is a right eigenvector of the transition matrix of the Markov chain $\markov(q,d)$.

We calculate explicitly the inverse of $M_0^{-1}$ as a matrix with rational entries:
\begin{equation*}
M_0^{-1}=\inv{q^2}\begin{pmatrix}
q & q(q-1) & -1 & -\frac{q-1}{2} & 0\\
0 & q^2 & 0 & -q & 0\\
0 & 0 & 1 & \frac{q-1}{2} & 0\\
0 & 0 & 0 & q & 0\\
0 & 0 & 0 & 0 & q^2
\end{pmatrix}.
\end{equation*}
Hence, we have
\begin{equation*}
M_1M_0^{-1}=\begin{pmatrix}
1 & 1 & 0 & 0 & 0\\
0 & 1 & 0 & 0 & 0\\
0 & 0 & 1 & 1 & 1\\
0 & 0 & 0 & 1 & 1\\
0 & 0 & 0 & 0 & 1
\end{pmatrix},
\end{equation*}
\begin{equation*}
M_1^2M_0^{-2}=\begin{pmatrix}
1 & q+1 & 0 & 0 & 2\\
0 & 1 & 0 & 0 & 1\\
0 & 0 & 1 & q+1 & \frac{(q+1)(q+2)}{2}\\
0 & 0 & 0 & 1 & q+1\\
0 & 0 & 0 & 0 & 1
\end{pmatrix}
\end{equation*}
and
\begin{equation*}
M_2^2M_0^{-2}=\begin{pmatrix}
1 & 2(q+1) & 0 & 0 & q+6\\
0 & 1 & 0 & 0 & 2\\
0 & 0 & 1 & 2(q+1) & (q+1)(2q+3)\\
0 & 0 & 0 & 1 & 2(q+1)\\
0 & 0 & 0 & 0 & 1
\end{pmatrix}.
\end{equation*}
In particular, $M_1M_0^{-1}$, $M_1^2M_0^{-2}$ and $M_2^2M_0^{-2}$ have all integer coefficients, and so they can always be reduced modulo $d$.

Consider now the matrices $M_\lambda$ modulo $d$. Each determinant is equal to $q^4$; since $q$ and $d$ are coprime, it follows that they are all invertible modulo $d$, and that their inverse is the reduction modulo $d$ of their rational inverses. Moreover, since $\mathrm{GL}_5(\insZ/d\insZ)$ is a finite group, each $M_\lambda$ has a finite order $h_\lambda$; hence, each state of Markov chain $\markov(q,d)$ is ergodic. Indeed, if $\mathbf{t}$ is reachable from $\mathbf{s}$, then $\mathbf{t}=\mathbf{s}M_{\lambda_1}\cdots M_{\lambda_k}$ for some $\lambda_1,\ldots,\lambda_k$; but then
\begin{equation*}
\mathbf{s}=\mathbf{t}M_{\lambda_k}^{h_{\lambda_k}-1}\cdots M_{\lambda_1}^{h_{\lambda_1}-1}
\end{equation*}
so $\mathbf{s}$ is reachable from $\mathbf{t}$.

\begin{comment}
Each matrix $M_\lambda$ is invertible, since its determinant is equal to $2q^4$ and $\gcd(d,2q)=1$; since $\mathrm{GL}_5(\insZ/d\insZ)$ is a finite group, each $M_\lambda$ has a finite order $h_\lambda$. Let $k$ be the least common multiple of all the $h_\lambda$, and consider the $k$-th iterated Markov chain $\markov(q,d)$, i.e., the Markov chain associated to $P^k$. Then, the probability of going from $\mathbf{s}$ to $\mathbf{t}$ is $\mu_k(\mathbf{s},\mathbf{t})/q^k$, where $\mu_k(\mathbf{s},\mathbf{t})$ is the number of $k$-tuples $(\lambda_1,\ldots,\lambda_k)$ such that $\mathbf{s}M_{\lambda_1}\cdots M_{\lambda_k}=\mathbf{t}$. Every state of this Markov chain is ergodic: indeed, if $\mathbf{t}$ is reachable from $\mathbf{s}$, then $\mathbf{t}=\mathbf{s}M_{\lambda_1}\cdots M_{\lambda_{kt}}$ for some $t$; but then
\begin{equation*}
\mathbf{s}=\mathbf{t}M_{\lambda_{kt}}^{k-1}\cdots M_{\lambda_1}^{k-1}
\end{equation*}
so $\mathbf{t}$ is reachable from $\mathbf{s}$ (note that the number of matrices is $kt(k-1)$, which is a multiple of $k$).
\end{comment}

We claim that, for every $a$ and $x$, the tuples $(1,a,0,0,x)$ and $(1,b,0,0,x)$ are communicating.

For every $\theta_w$ and every $x$, we have the three equalities
\begin{equation*}
\begin{pmatrix}1 & \theta_w & 0 & 0 & x\end{pmatrix}M_1M_0^{-1}=\begin{pmatrix}1 & \theta_w+1 & 0 & 0 & x\end{pmatrix},
\end{equation*}
\begin{equation*}
\begin{pmatrix}1 & 0 & 0 & 0 & x\end{pmatrix}M_1^2M_0^{-2}=\begin{pmatrix}1 & q+1 & 0 & 0 & x+2\end{pmatrix}
\end{equation*}
and
\begin{equation*}
\begin{pmatrix}1 & 0 & 0 & 0 & x\end{pmatrix}M_2^2M_0^{-2}=\begin{pmatrix}1 & 2(q+1) & 0 & 0 & x+q+6\end{pmatrix}.
\end{equation*}
Since $q\geq 2$, we can always use $M_0$ and $M_1$: hence, the first equality implies that, for every $x$, the 5-tuples $(1,a,0,0,x)$ and $(1,b,0,0,x)$ are communicating, while the first and the second imply that $(1,0,0,0,x)$ and $(1,a,0,0,x+2)$ are communicating for every $a$ and every $x$ (in particular, $a=0$). 

If $d$ is odd, this implies that $(1,0,0,0,x)\leftrightarrow(1,0,0,0,y)$ for every $x,y\in\insZ/d\insZ$, and thus that $(1,a,0,0,x)$ and $(1,b,0,0,y)$ are communicating whatever $a$ and $b$ are.

If $d$ is even, then $q\geq 3$ (since $\gcd(d,q)=1$): hence, we can also use $M_2$, obtaining that $(1,0,0,0,x)\leftrightarrow(1,2(q+1),0,0,x+q+6)$, and so $(1,0,0,0,x)\leftrightarrow(1,a,0,0,x+q+6)$ for every $a$. Hence, $(1,0,0,0,x)$ is communicating with 
\begin{equation*}
(1,a,0,0,x+2z_1+(q+6)z_2)
\end{equation*}
for every choice of $z_1,z_2\inN$. However, $q$ is odd, and thus $\gcd(2,q+6,d)=1$: hence, $2z_1+(q+6)z_2$ can be equal to any $s\in\insZ/d\insZ$. Therefore, $(1,0,0,0,x)\leftrightarrow(1,a,0,0,y)$ for every $a$ and $y$, as claimed.

Therefore, for every $d$, we have $\delta(1,a,0,0,x)=\delta(1,b,0,0,y)$ for each $a,b,x,y$. However,
\begin{equation*}
\sum_{x=0}^{d-1}\delta(1,\theta_w,0,0,x)=1;
\end{equation*}
hence, for every $a$ and every $x$ we have $\delta(1,a,0,0,x)=1/d$, as claimed.
\end{proof}

\section{Algebraic interpretation}\label{sect:Int}
Let $D$ be an integral domain with quotient field $K$. The set of \emph{integer-valued polynomials} on $D$ is
\begin{equation*}
\Int(D):=\{f\in K[X]\mid f(D)\subseteq D\}.
\end{equation*}
The set $\Int(D)$ is always an integral domain contained between $D[X]$ and $K[X]$. There are two sequences of $D$-modules associated to $\Int(D)$: the first is formed by the \emph{characteristic ideals} $\caratt_n:=\caratt_n(D)$, defined as the union of $(0)$ with the leading coefficients of the polynomials of $\Int(D)$ of degree $n$; the second contains the modules of the form
\begin{equation*}
\Int_n(D):=\{f\in\Int(D)\mid \deg f\leq n\}.
\end{equation*}
If $D$ is a Dedekind domain, these two sequences are linked by the relation \cite[Corollary II.3.6]{IntD}
\begin{equation}\label{eq:Intn-simeq}
\Int_n(D)\simeq\bigoplus_{k=0}^n\caratt_k\simeq D^n\oplus\prod_{k=0}^n\caratt_k.
\end{equation}

For any maximal ideal $\mathfrak{m}$ of $D$, let $N(\mathfrak{m})$ be the \emph{norm} of $\mathfrak{m}$, that is, the cardinality of $D/\mathfrak{m}$; for any $q\inN$, let $\Pi_q$ be the product of the maximal ideals of norm $q$. The subgroup of the class group generated by the $\Pi_q$ is called the \emph{P\'olya-Ostrowski group} of $D$, and is denoted by $\polya(D)$. Several papers studied $\polya(D)$ when $D$ is an integral extension of $\insZ$, with special focus on looking for which $D$ the group $\polya(D)$ is trivial (in this case, the quotient field $K$ of $D$ is said to be a \emph{P\'olya field}) \cite{zantema,leriche,biquadratic-polya}: this happens if and only if $\Int(D)$ has a \emph{regular basis}, i.e., a basis $\{f_0,f_1,\ldots,\}$ over $D$ such that $\deg f_i=i$ for every $i$. For example, every cyclotomic extension of $\insQ$ is a P\'olya field \cite[Proposition II.4.3]{IntD}.

The (classes of the) characteristic ideals of $D$ naturally belong to $\polya(D)$, and by \cite[Proposition II.3.9]{IntD} we have
\begin{equation*}
\caratt_n=\prod_{q=2}^n\Pi_q^{-w_q(n)}.
\end{equation*}
On the other hand, the modules $\Int_n(D)$ does not belong, by themselves, to $\polya(D)$, for the trivial reason that they are not fractional ideals of $D$. However, by \eqref{eq:Intn-simeq}, we can write $\Int_n(D)$ as the direct sum $D^n\oplus\widehat{\Int_n}(D)$, where
\begin{equation*}
\widehat{\Int_n}(D):=\prod_{k=1}^n\caratt_n(D)
\end{equation*}
is a fractional ideal of $D$; by construction, the isomorphism class of $\widehat{\Int_n}(D)$ belongs to $\polya(D)$. Moreover, since $D$ is a Dedekind domain, $\widehat{\Int_n}(D)$ is a projective module of rank 1, and thus $\Int_n(D)$ is free if and only if $\widehat{\Int_n}(D)$ is free  \cite[Theorem 4.11]{lam-serre}. Applying again \eqref{eq:Intn-simeq}, we see that
\begin{equation*}
\widehat{\Int_n}(D)\simeq\prod_{q=2}^n\Pi_q^{-u_q(n)}
\end{equation*}
for every $n\inN$.

Therefore, we have two maps $\insN\longrightarrow\polya(D)$ given by
\begin{equation*}
n\mapsto[\caratt_n(D)]\quad\text{and}\quad n\mapsto[\widehat{\Int_n}(D)],\end{equation*}
and studying how many times $\caratt_n(D)$ and $\widehat{\Int_n}(D)$ are isomorphic to a specific module is essentially equivalent to studying the limit distribution of these maps in $\polya(D)$. Elliott \cite{elliott-polya} conjectured that the density of natural numbers such that $\Int_n(D)$ is free exists and is rational, and moreover that it is always at least $1/|\polya(D)|$. He also makes several conjectures for more specific cases, mostly expressed in terms of multisets.

Suppose now that there is a unique $q$ such that $\Pi_q$ is not a principal ideal of $D$, and let $d:=|\polya(D)|$. Then, $\polya(D)\simeq\insZ/d\insZ$ is a cyclic group and $[\Pi_q]$ is a generator; moreover,
\begin{equation*}
\caratt_n\simeq\Pi_q^{-w_q(n)\bmod d}\quad\text{and}\quad\widehat{\Int_n}(D)\simeq\Pi_q^{-u_q(n)\bmod d}.
\end{equation*}
Therefore, the distribution of the maps $n\mapsto[\caratt_n(D)]$ and $n\mapsto[\widehat{\Int_n}(D)]$ in $\polya(D)$ is determined by the distribution of $n\mapsto w_q(n)$ and $n\mapsto u_q(n)$ modulo $d$, and we can simply translate the results in Sections \ref{sect:wq} and \ref{sect:uq} to this context; the definitions of limit distribution and of being uniformly distributed in $\polya(D)$ is analogous to the ones after Definition \ref{defin:density}.
\begin{prop}\label{prop:dedekind}
Let $D$ be a Dedekind domain, and suppose that $\Pi_n$ is nonprincipal only for $n=q$; let $d:=|\polya(D)|$.
\begin{enumerate}[(a)]
\item The map $n\mapsto[\caratt_n(D)]$ is uniformly distributed in $\polya(D)$.~
\item If, for every $\mathbf{s}$, the density $\delta(q,d;\mathbf{s})$ exists, then for every $g\in\polya(D)$ the density of $n$ such that $[\widehat{\Int_n}(D)]=g$ is rational.
\item If $d|q$, the density of $n$ such that $\Int_n(D)\simeq\Pi_q^x\oplus D^n$ is equal to $\displaystyle{\inv{d^2}\sum_{f|\gcd(x,d)}f\varphi\left(\frac{d}{f}\right)}$.
\item If $q$ and $d$ are coprime and, for every $\mathbf{s}\inZ^5$, the density $\delta(q,d;\mathbf{s})$ exists, then the map $n\mapsto[\widehat{\Int_n}(D)]$ is uniformly distributed in $\polya(D)$.
\end{enumerate}
\end{prop}
\begin{proof}
The four statements are the translation, respectively, of Corollary \ref{cor:wqthetan}, Theorem \ref{teor:rational}, Theorem \ref{teor:d|q} and Theorem \ref{teor:dqcoprimi}.
\end{proof}

Suppose $D$ is the ring of integers of a number field $K$. Some examples in which $D$ satisfies the hypotheses of the previous proposition are given in \cite[Examples 7.3 and 7.4]{elliott-polya}. Other examples can be constructed using \cite[proof of Proposition II.4.2]{IntD}: if $K$ is Galois over $\insQ$, then $\Pi_q$ is principal for every $q=p^r$ such that $p$ is not ramified in $K$. Since $p$ is ramified if and only if it divides the discriminant \cite[Chapter II, Corollary 2.12]{neukirch}, Proposition \ref{prop:dedekind} can be applied if the discriminant of $K$ is a prime power. Unfortunately, the simplest cases in which this happens (the quadratic fields $K=\insQ(\sqrt{p})$, where $|p|$ is a prime and $p\equiv 1\bmod 4$, and the cyclotomic extensions $K=\insQ(\zeta_p)$ for $p$ prime) are also P\'olya fields \cite[Proposition II.4.3 and Corollary II.4.5]{IntD}, and thus $\polya(D)$ is actually trivial; other examples where the discriminant is a prime power (with unknown P\'olya-Ostrowski group) are collected in \cite{jones-ramified}.

When there is more than one $\Pi_q$ which is nonprincipal, it is necessary to study the density of the $n$ that are simultaneous solution to the equations
\begin{equation*}
\theta_u^{(i)}u_{q_i}(n)+\theta_w^{(i)}w_{q_i}(n)+\theta_2^{(i)}\frac{n(n+1)}{2}+\theta_1^{(i)}n+\theta_0^{(i)}\equiv 0\bmod d_i
\end{equation*}
for $i\in\{1,\ldots,k\}$, where $q_1,\ldots,q_k$ and $d_1,\ldots,d_k$ are arbitrary integers (and $q_1,\ldots,q_k$ are pairwise different). This is essentially a problem in determining how much these equations are correlated; it is reasonable to think that (under the obvious hypothesis that at least one between $\theta_u^{(i)}$ and $\theta_w^{(i)}$ is nonzero modulo $d_i$, for each $i$) such equations are actually independent, so that the density of the solutions of all the equations is determined by the densities of the solutions of the single equations. The major obstruction seems to be the problem of understanding the behaviour of $u_{q_1}(aq_2+\lambda)$ and $w_{q_1}(aq_2+\lambda)$ when $q_1\neq q_2$.

\section*{Acknowledgements}
The author wishes to thank Francesco Pappalardi and the referee for their suggestions.

\end{document}